\documentclass[10pt,twoside]{amsart}
\usepackage{latexsym,amssymb,graphicx}
\usepackage{color}
\usepackage{mathrsfs}
\usepackage{mathabx}
\usepackage[all]{xy}

\usepackage{oldgerm}

\let\oldtocsection=\tocsection

\let\oldtocsubsection=\tocsubsection

\let\oldtocsubsubsection=\tocsubsubsection

\renewcommand{\tocsection}[2]{\hspace{0em}\oldtocsection{#1}{#2}}
\renewcommand{\tocsubsection}[2]{\hspace{1em}\oldtocsubsection{#1}{#2}}
\renewcommand{\tocsubsubsection}[2]{\hspace{2em}\oldtocsubsubsection{#1}{#2}}

\newtheorem{theorem}{Theorem}[section]
\newtheorem{lemma}[theorem]{Lemma}
\newtheorem{corollary}[theorem]{Corollary}
\newtheorem{proposition}[theorem]{Proposition}

\theoremstyle{definition}

\newtheorem{example}[theorem]{Example}

\newtheorem{remark}[theorem]{Remark}

\numberwithin{equation}{section}

\newcommand{\Z}{\mathbb{Z}}

\newcommand{\A}{\mathbb{A}}

\newcommand{\F}{\mathbb{F}}

\newcommand{\eps}{\varepsilon}

\newcommand{\rk}{\operatorname{rk}}

\newcommand{\mylangle}{\langle\hspace{-.5ex}\langle}
\newcommand{\myrangle}{\rangle\hspace{-.5ex}\rangle}

\usepackage{mathtools}
\newcommand{\defeq}{\vcentcolon=}

\usepackage{color}
\newcommand{\red}[1]{{#1}}
\newcommand{\m}{\textswab{m}}
\renewcommand{\b}{\textswab{b}}
\newcommand{\q}{\textswab{q}}

\title[Presentation of $GW(R)$]{On the presentation of the Grothendieck-Witt group of symmetric bilinear forms over local rings}
 \author{Robert Rogers and Marco Schlichting}

 \address{Mathematics Institute,
Zeeman Building,
University of Warwick,
Coventry CV4 7AL, UK} 

\thanks{}

\email{m.schlichting@warwick.ac.uk\\ Robert.Rogers@warwick.ac.uk}



\begin{document}
\bibliographystyle{alpha}

\begin{abstract}
We prove a Chain Lemma for inner product spaces over commutative local rings $R$ with residue field other than $\F_2$ and use this to show that the usual presentation of the Grothendieck-Witt group of symmetric bilinear forms over $R$ as the zero-th Milnor-Witt $K$-group holds provided the residue field of $R$ is not $\F_2$.
 \end{abstract}

\maketitle

\tableofcontents

\section{Introduction}

Extending work of Witt \cite{Witt} to include the case of characteristic $2$ fields, Milnor-Husemoller prove in \cite[Lemma IV.1.1]{MilnorHusemoller} that the Witt group $W(F)$ of inner product spaces, aka non-degenerate symmetric bilinear forms, of a field $F$ is additively generated by elements $\langle a \rangle$, with $a\in F^*$, subject to the following three relations.
\begin{enumerate}
\item
\label{item:Milnor1}
For all $a,b\in F^*$ we have $\langle a^2b \rangle =\langle b \rangle$.
\item
\label{item:Milnor2}
For all $a\in F^*$ we have $\langle a \rangle + \langle -a \rangle = 0$.
\item
\label{item:Milnor3}
For all $a,b, a+b\in F^*$ we have $\langle  a \rangle  + \langle b \rangle = \langle a+b\rangle + \langle (a+b)ab\rangle$.
\end{enumerate}
From this, one readily obtains a presentation of the Grothendieck-Witt group $GW(F)$ of $F$ with the same generators and relations (\ref{item:Milnor1}), (\ref{item:Milnor2}'), (\ref{item:Milnor3}) where:
\begin{itemize}
\item[(\ref{item:Milnor2}')]
For all $a\in F^*$ we have $\langle a \rangle + \langle -a \rangle = \langle 1 \rangle + \langle -1\rangle$.
\end{itemize}
The goal of this paper is to generalise these presentations to commutative local rings $(R,\m,F)$.
In fact, we will show in Theorem \ref{thm:presentation} and Corollary \ref{cor:MilnorPresentation} below that the same presentation holds for $GW(R)$ and for $W(R)$ as long as the residue field $F=R/\m$ of the local ring $R$ satisfies $F\neq \F_2$. 
If the residue field is $\F_2$, then there are counter-examples; see Proposition \ref{prop:tildeInotI}.
It seems that our results are new when the residue field $F$ has characteristic $2$ or when $R\neq F=\F_3$.

\begin{remark}
The abelian group with generators $\langle a\rangle$, $a\in R^*$, and relations (\ref{item:Milnor1}), (\ref{item:Milnor2}'), (\ref{item:Milnor3}) (and $R$ in place of $F$) is also known as the zero-th Milnor-Witt $K$-group $K^{MW}_0(R)$ of $R$ \cite{morel:book}, \cite{GilleEtAl}, \cite{myEuler}.
The presentation of $GW(R)$ as the zeroth Milnor-Witt $K$-group has become important in applications of $\A^1$-homotopy theory \cite{morel:book}, \cite{AsokFasel:ICM} and the homology of classical groups \cite{myEuler} where the sheaf of Milnor-Witt $K$-groups plays a paramount role.
To date, the lack of understanding of the relation between Milnor-Witt $K$-theory and Grothendieck-Witt groups when $\operatorname{char}(F)=2$ is the reason that many results are only known away from characteristic $2$.
This paper therefore is part of the effort to establish these applications also in characteristic $2$ and in mixed characteristic.
\end{remark}
\vspace{1ex}

\noindent
{\bf Statement of results}. 
To state our results, recall that an {\em inner product space} 
over a commutative ring $R$ is a finitely generated projective $R$-module $V$ equipped with a non-degenerate symmetric $R$-bilinear form $\b: V\times V \to R$; see \cite{MilnorHusemoller}.
When $R$ is local, then $V$ is free of some finite rank, say $n$.
In that case, an {\em orthogonal basis} of $V$ is a basis $v_1,...,v_n$ of $V$ such that $\b (v_i,v_j)=0$ for $i\neq j$.
Note that if the residue field of $R$ has characteristic $2$, an inner product space over $R$ need not have an orthogonal basis.
Nevertheless, we prove in Proposition \ref{prop:GWgens} (\ref{prop:GWgens:item3}) that {\em stably} every inner product space over a local commutative ring $R$ has an orthogonal basis.
Two orthogonal bases $B,C$ of $V$ are called {\em chain equivalent}, written $B\approx C$, or $B\approx_R C$ to emphasise the ring $R$,
if there is a sequence $B_0,B_1,...,B_r$ of orthogonal bases of $V$ such that $B_0=B$ and $B_r=C$, and $B_{i-1}\cap B_{i}$ has cardinality at least $n-2$ for $i=1,...,r$.
Our first result is the following.

\begin{theorem}[Chain Lemma]
\label{thm:ChainLemma}
Let $(R,\m,F)$ be a commutative local ring with residue field $F\neq \F_2$.
Let $V$ be an inner product space over $R$.
Then any two orthogonal bases of $V$ are chain equivalent.
\end{theorem}

Of course, this is vacuous if $V$ has no orthogonal basis. 
Theorem \ref{thm:ChainLemma} was previously known when $R$ is a field of characteristic not $2$ \cite[Satz 7]{Witt}, \cite[Theorem I.5.2]{Lam:book}, and the local case easily reduces to the field case; see Lemma \ref{lem:ChainReductionRtoF}.
The Theorem does not hold when $F=\F_2$; see Remark \ref{rmk:NoF2Chain} and Lemma \ref{lem:ChainReductionRtoF}.
The proof of Theorem \ref{thm:ChainLemma} is given in Section \ref{sec:ChainLemma}.
\vspace{1ex}

We let $GW(R)$ be the Grothendieck-Witt ring of non-degenerate symmetric bilinear forms over $R$, 
that is, the Grothendieck group associated with the abelian monoid of isomorphism classes of inner product spaces over $R$ with orthogonal sum as monoid operation \cite{Knebusch}, \cite{Sah}, \cite{MilnorHusemoller}, \cite{myHermKex}.
The ring structure is induced by the tensor product of inner product spaces.
For $a\in R^*$, we denote by $\langle a \rangle_{\Z}$ the $\Z$-basis element of the group ring $\Z[R^*]$ corresponding to $a\in R^*$, and by $\langle a \rangle$ the rank $1$ inner product space $\b( x,y) = axy$, $x,y\in V=R$.
We have elements $\mylangle a\myrangle_{\Z}=1-\langle a \rangle_{\Z}$ and $h_{\Z}=\langle 1 \rangle_{\Z} + \langle -1 \rangle_{\Z}$ in $\Z[R^*]$ and $\mylangle a\myrangle=1-\langle a \rangle$ and $h=\langle 1 \rangle + \langle -1 \rangle$ in $GW(R)$.
We may write $\langle a \rangle$, $\mylangle a \myrangle$ and $h$ in place of $\langle a \rangle_{\Z}$, $\mylangle a \myrangle_{\Z}$ and $h_{\Z}$ if their containment in $\Z[R^*]$ is understood.
Note that we have a ring homomorphism
\begin{equation}
\label{eqn:ZRGWringhomo}
\pi: \Z[R^*] \longrightarrow GW(R): \langle a \rangle_{\Z} \mapsto \langle a \rangle\end{equation}
which sends $\mylangle a \myrangle_{\Z}$ and $h_{\Z}$ to $\mylangle a \myrangle$ and $h$.
Our main result is the following which asserts that this ring homomorphism is surjective with kernel the ideal generated by three types of relations. 

\begin{theorem}[Presentation of $GW(R)$]
\label{thm:presentation}
Let $(R,\m,F)$ be a commutative local ring with residue field $F\neq \F_2$. 
Then the Grothendieck-Witt ring $GW(R)$ of inner product spaces over $R$ is the quotient ring of the integral group ring $\Z[R^*]$ of the group $R^*$ of units of $R$ modulo the following relations:
\begin{enumerate}
\item
\label{thm:presentation:SquareTriviality}
For all $a\in R^*$ we have $\mylangle a^2 \myrangle =0$.
\item
\label{thm:presentation:ScalarInv}
For all $a\in R^*$ we have $\mylangle a \myrangle  \cdot h = 0$.
\item
\label{thm:presentation:SteinbergRln}
(Steinberg relation)
For all $a,1-a\in R^*$ we have $\mylangle  a \myrangle  \cdot \mylangle  1-a \myrangle =0$.
\end{enumerate}
\end{theorem}

In the context of Witt and Grothendieck-Witt groups, the Steinberg relation is also called Witt relation.

\begin{remark}
\label{rmk:IrrelvRlns}
If the residue field $F$ of $R$ satisfies $F\neq \F_2,\F_3$ and we impose only the Steinberg relation (\ref{thm:presentation:SteinbergRln}) in Theorem \ref{thm:presentation}, 
then imposing relation
(\ref{thm:presentation:SquareTriviality}) is equivalent to imposing relation (\ref{thm:presentation:ScalarInv}); see Lemma \ref{lem:GWtoKMWrlns} (\ref{lem:GWtoKMWrlns:item2}) below.
In particular, if the residue field is not $\F_2,\F_3$, then $GW(R)$ is the ring quotient of the group ring $\Z[R^*/(R^*)^2]$ of the group of unit square classes modulo the Steinberg relation (\ref{thm:presentation:SteinbergRln}).
When $R=F$ is any field, including $F=\F_2, \F_3$, we can dispense with the relation  (\ref{thm:presentation:ScalarInv}) as well and obtain the presentation of $GW(F)$ as the quotient of the group ring $\Z[R^*/(R^*)^2]$ modulo the Steinberg relations.
Indeed, if $R=\F_3$, relations (\ref{thm:presentation:SquareTriviality}) and (\ref{thm:presentation:ScalarInv}) are vacuous and 
if $R=\F_2$, all three relations (\ref{thm:presentation:SquareTriviality}),  (\ref{thm:presentation:ScalarInv}) and (\ref{thm:presentation:SteinbergRln}) are vacuous but the map $\pi:\Z=\Z[R^*]\to GW(R)$  \red{in (\ref{eqn:ZRGWringhomo})} is already an isomorphism.
\end{remark}

Theorem \ref{thm:presentation} was previously known for $R$ a field (including $\F_2$) \cite{MilnorHusemoller}, and for commutative local rings with residue field $F$ of characteristic not two as long as  $F\neq \F_3$ \cite[Theorem 2.2]{Gille:Semilocal}.
The theorem does not hold for local rings with residue field $\F_2$, in general; see Proposition \ref{prop:tildeInotI}.
The proof of Theorem \ref{thm:presentation}
is in Section \ref{sec:Presentation}, Corollary \ref{cor:KMWGW}.
\vspace{1ex}

Since the Witt ring $W(R)$ is the quotient of the Grothendieck-Witt ring $GW(R)$ modulo the ideal generated by $h=1+\langle -1\rangle$, we obtain the following from Theorem \ref{thm:presentation} generalising \red{the} presentation \red{\cite[Lemma IV.1.1]{MilnorHusemoller}} from fields to commutative local rings.

\begin{corollary}
\label{cor:MilnorPresentation}
Let $(R,\m,F)$ be a commutative local ring with residue field $F\neq \F_2$.
Then the Witt group $W(F)$ of inner product spaces of $R$ is additively generated by elements $\langle a \rangle$, with $a\in R^*$, subject the following three relations.
\begin{enumerate}
\item
For all $a,b\in R^*$ we have $\langle a^2b \rangle =\langle b \rangle$.
\item
For all $a\in R^*$ we have $\langle a \rangle + \langle -a \rangle = 0$.
\item
For all $a,b, a+b\in R^*$ we have $\langle  a \rangle  + \langle b \rangle = \langle a+b\rangle + \langle (a+b)ab\rangle$.
\end{enumerate}
\end{corollary}

\vspace{1ex}

\noindent
\red{{\it Acknowledgements}.
We would like to thank the referee for a careful reading of the manuscript.
Robert Rogers would like to thank the Institute of Mathematics at the University of Warwick for providing financial support in the form of a URSS grant while the research was carried out.}

\section{The Chain Lemma}
\label{sec:ChainLemma}

All rings in this article are assumed commutative.
For an inner product space \red{$(V,\b)$} over a ring $R$, we write $\q:V\to R$ for the associated quadratic form defined by \red{$\q(x)=\b(x,x)$} for $x\in V$.
We call an element $v\in V$ {\em anisotropic} if $\q(v)\in R^*$.
Note that for an orthogonal basis $(u_1,...,u_n)$ of $V$, every $u_i$ is anisotropic, $i=1,...,n$.
For units $a_1,...,a_n \in R^*$, we denote by $\langle a_1,...,a_n\rangle = \langle a_1 \rangle + \cdots + \langle a_n \rangle = \langle a_1 \rangle \oplus \cdots \oplus \langle a_n \rangle$ the inner product space which has \red{an} orthogonal basis $u_1,...,u_n$ with $\q(u_i)=a_i$ for $i=1,...,n$.

Our first  goal is to show in Lemma \ref{lem:ChainReductionRtoF} below that the Chain Lemma (Theorem \ref{thm:ChainLemma}) for a  local ring is equivalent to the Chain Lemma for its residue field.

\begin{lemma}
\label{ChainLemStrat:item1}
Let $(R,\m,F)$ be a local ring, $\eps \in \m$, and let $V$ be an inner product space over $R$. 
If $B_1 = (u_1,...,u_n)$ is an orthogonal basis of $V$, then so is $B_2 = (u_1+\eps u_2, u_2-\eps \q(u_2)\q(u_1)^{-1}u_1,u_3,...,u_n)$.
Moreover, we have $B_1=B_2 \mod \m$ and $B_1 \approx_R B_2$.
\end{lemma}

\begin{proof}
Since $\eps \in \m$, we have  $B_1=B_2 \mod \m$, and $B_2$ is a basis since $B_1$ is.
Orthogonality is checked directly. Since $B_1$ and $B_2$ differ in only two terms, they are chain equivalent, by definition.
\end{proof}

\begin{lemma}
\label{lem:EqualModM}
 Let $(R,\m,F)$ be a local ring, and let $V$ be an inner product space over $R$. 
 If $B_1=(u_1,...,u_n)$ and $B_2=(v_1,...,v_n)$ are orthogonal bases of $V$ such that $B_1= B_2 \mod \m$, then $B_1 \approx_R B_2$.
\end{lemma}

\begin{proof}
The proof is by induction on $n\geq 1$.
By the definition, for $n=1$ and $n=2$ any two orthogonal bases are chain equivalent. In particular, the claim is true for $n=1,2$.
For $n> 2$,
we claim that $(u_1,\red{u_2},...,u_n) \approx_R (v_1,u'_2,...,u'_n)$ for some $u'_2,...,u'_n\in V$ such that $u'_i=u_i\mod \m$, \red{$i=2,...,n$}.
Then the induction hypothesis applied to the two orthogonal bases $(u'_2,...,u'_n)$ and $(v_2,...,v_n)$ of the non-degenerate subspace $v_1^{\perp}$ of $V$ yields $(u_1,\red{u_2},...,u_n) \approx_R (v_1,u'_2,...,u'_n) \approx_R(v_1,v_2,...,v_n)$.
To prove the claim, note that $v_1=u_1+\eps_1u_1+\eps_2u_2+\cdots +\eps_nu_n$ for some $\eps_i\in \m$ since $u_1=v_1\mod \m$.
For $i=0,...,n$, set $u_1^{(i)} = u_1+\eps_1u_1+\eps_2u_2+\cdots +\eps_iu_i$.
Then $u_1^{(0)}=u_1$ and $u_1^{(n)}=v_1$.
For $i=2,...,n$, we apply Lemma \ref{ChainLemStrat:item1} recursively to the pair $(u_1^{(i-1)},u_i)$ to find $u'_i\in V$ such that $u'_i=u_i\mod \m$ and 
$$(u_1,\red{u_2},...,u_n) \approx_R (u_1^{(1)},u_2,...,u_n) \approx_R (u_1^{(i)},u'_2,...,u'_i,u_{i+1},...,u_n)$$
where the first $\approx_R$ is the case $n=1$.
\end{proof}

\begin{lemma}
    \label{ChainLemStrat:item3pre}
Let $(R,\m,F)$ be a local ring, and let $V$ be an inner product space over $R$. 
Any orthogonal basis $\bar{u}=(\bar{u}_1,...,\bar{u}_n)$ of $V_F=V\otimes_RF$ is the image mod $\m$ of an orthogonal basis $u=(u_1,...,u_n)$ of $V$, called lift of $\bar{u}$.
If two orthogonal bases $\bar{u}$, $\bar{v}$ of $V_F$ differ by at most two places, then there are lifts $u$ and $v$ of $\bar{u}$ and $\bar{v}$ which differ in at most two places.
\end{lemma}

\begin{proof}
Choose any lift $u_1$ of $\bar{u}_1$ inside $V$, then any lift $u_2$ of $\bar{u}_2$ inside $u_1^{\perp}\subset V$, then any lift $u_3$ of $\bar{u}_3$ inside $\{u_1,u_2\}^{\perp}\subset  V$... 
This yields a lift $u$ of $\bar{u}$.
Assume $\bar{u}=(\bar{u}_1,\bar{u}_2, \bar{u}_3,...,\bar{u}_n)$ and $\bar{v}=(\bar{v}_1,\bar{v}_2, \bar{u}_3,...,\bar{u}_n)$.
Let $u=(u_1,...,u_n)$ be a lift of $\bar{u}$. 
Let $(v_1,v_2)$ be a lift of $(\bar{v}_1,\bar{v}_2)$ inside $\{u_3,...,u_n\}^{\perp}$.
Then we can choose $v=(v_1,v_2,u_3,...,u_n)$ as lift of $\bar{v}$.
\end{proof}

For two orthogonal bases $B,C$ of an inner product space $V$ over a local ring $(R,\m,F)$, we write $B\approx_F C$ if the images of $B$ and $C$ in $V_F=V \otimes_RF$ are chain equivalent over $F$.
The following shows that the Chain Lemma (Theorem \ref{thm:ChainLemma}) for a local ring is equivalent to the Chain Lemma for its residue field.

\begin{lemma}
\label{lem:ChainReductionRtoF}
Let $(R,\m,F)$ be a local ring and $V$ an inner product space over $R$.
For two orthogonal bases $B$, $C$ of $V$, if $B\approx_F C$, then $B \approx_R C$.
\end{lemma}

\begin{proof}
Choose a sequence $\bar{B}_i$, $i=0,...,r$ of orthogonal bases of $V_F$ such that $\bar{B}_0$ and $\bar{B}_r$ are the images of $B$ and $C$ in $V_F$ and $\bar{B}_i$ differs from $\bar{B}_{i+1}$ in at most two places, $i=0,...,r-1$.
By Lemma \ref{ChainLemStrat:item3pre}, for $i=0,...,r-1$ we can choose lifts $B_i$, $C_{i+1}$ of $\bar{B}_i$ and $\bar{B}_{i+1}$ such that $B_i$ and $C_{i+1}$ differ in at most two places.
By Lemma \ref{lem:EqualModM}, we have $B\approx_R B_0$, $B_i\approx_RC_i$ for $i=1,...,r-1$ and $C_r\approx_RC$.
Hence,
$$B\approx_R B_0 \approx_R C_1 \approx_RB_1 \approx_R C_2 \approx_RB_2 \approx_R C_3 \approx_R \cdots \approx_R C_r \approx_RC.$$ 
\end{proof}

Our next goal is to prove in Theorem \ref{thm:FieldChainLem} the Chain Lemma (Theorem \ref{thm:ChainLemma}) for infinite fields of characteristic $2$.
We will make frequent use of the following.

\begin{lemma}
\label{lem:ElementChains}
Let $n\geq 2$ be an integer, and 
let $u=(u_1,...,u_n)$ be an orthogonal basis of an inner product space $V$ of rank $n$ over a field $F$.
Let $v_1=a_1u_1+ \cdots +a_nu_n$, where $a_1,...,a_n\in F$.
If for all $2\leq r \leq n$, the partial linear combination $v_1^{(r)} = a_1u_1+\cdots +a_ru_r$ is anisotropic, then $v_1=v_1^{(n)}$ can be extended to an orthogonal basis $v=(v_1,...,v_n)$ of $V$ such that $u\approx_F v$.
\end{lemma}

\begin{proof}
Choose $v_2$ to be a generator of the orthogonal of $v_1^{(2)}$ inside $Fu_1\perp Fu_2$.
Then $u \approx (v_1^{(2)},v_2,u_3,...,u_n)$.
For an integer $r$ with $2 \leq r <n$, 
assume we have constructed elements $v_2,...,v_{r}\in V$ such that $(v_1^{(r)}, v_2,...,v_{r}, u_{r+1},...,u_n)$ is an orthogonal basis of $V$ that is chain equivalent to $u$.
Note that $v_1^{(r+1)}$ is an anisotropic vector in $Fv_1^{(r)}\perp Fu_{r+1}$.
Choose $v_{r+1}$ to be a generator of the orthogonal complement $(v_1^{(r+1)})^{\perp}$ of $Fv_1^{(r+1)}$ inside $Fv_1^{(r)}\perp Fu_{r+1}$.
Then
$$u\approx (v_1^{(r)}, v_2,...,v_r, u_{r+1},...,u_n) \approx (v_1^{(r+1)}, v_2,...,v_{r+1}, u_{r+2},...,u_n).$$
By induction on $r$, we obtain the case $r=n$ which is the statement of the lemma.
\end{proof}

\begin{theorem}
\label{thm:FieldChainLem}
Let $F$ be a field of characteristic 2, and let
 $V$ be an inner product space over $F$.
If $F$ is finite, assume that $\dim_FV=3$.
Then any two orthogonal bases of $V$ are chain equivalent.
\end{theorem}

\begin{proof}
Assume first that $F\neq \F_2$.
We proceed by induction on $n=\dim_F V\geq 0$.
For $n=0,1,2$, there is nothing to prove. 
If $F$ is finite, assume $n=3$, otherwise let $n\geq 3$.
For an orthogonal basis $u=(u_1,u_2,...,u_n)$ of $V$, let 
$C(u)\subset V$ be the set of all vectors $\alpha_1 u_1 + \alpha_2 u_2 + ... + \alpha_n u_n \in V$, \red{with $\alpha_i\in F$}, such that
 $$\alpha_1^2\q(u_1) + \alpha_2^2\q(u_2) + \cdots + \alpha_r^2\q(u_r) \neq 0\hspace{3ex}\text{for all}\ \ r=2,...,n.$$
Let $v=(v_1,v_2,...,v_n)$ be another orthogonal basis of $V$ and consider the corresponding set $C(v)$.
By Lemma \ref{lem:q14sols} below, the intersection $C(u)\cap C(v)$ is non-empty.
Thus, we can choose a vector $u'_1=v'_1\in C(u)\cap C(v)$.
By Lemma \ref{lem:ElementChains} we can extend $u'_1=v'_1$ to orthogonal bases $u'=(u'_1,u'_2,...,u'_n)$ and $v'=(v'_1,v'_2,...,v'_n)$ of $V$ such that $u\approx u'$ and $v\approx v'$.
Now $(u'_2,...,u'_n)$ and $(v'_2,...,v'_n)$ are orthogonal bases of $(u'_1)^{\perp} = (v'_1)^{\perp}$ and thus 
$(u'_2,...,u'_n)\approx (v'_2,...,v'_n)$ by the induction hypothesis.
In particular, $u'\approx v'$ since $u'_1=v'_1$, and
we have proved $u\approx u'\approx v' \approx v.$

For $F=\F_2$ there is only one inner product space $V$ of dimension $3$, namely $\langle 1,1,1\rangle$; see for instance Proposition \ref{prop:GWgens} below.
The only anisotropic vectors of $V$ are the vectors of the standard orthonormal basis $e_1$, $e_2$, $e_3$, and $e=e_1+e_2+e_3$.
The vector $e$ cannot be extended to an orthogonal basis since every vector in its orthogonal complement $e^{\perp} \subset V$ is isotropic.
Thus, the only orthogonal basis of $V$ is $e_1,e_2,e_3$ and the theorem trivially holds.
\end{proof}

\begin{lemma}
\label{lem:q14sols}
Let $n,r\geq 1$ be integers, and let $F$ be a field of characteristic $2$. 
Let $V=F^n$ and let $\q_1,...,\q_r$ be diagonalisable non-trivial homogeneous quadratic forms on $V$.
If $|F|\geq r$, then there is $v\in V$ such that $\q_i(v)\neq 0$ for $i=1,...,r$.    
\end{lemma}

\begin{proof}
We proceed by induction on $r\geq 1$. 
If $r=1$ the quadratic form $\q_1$ can be written as $\alpha_1 x_1^2 + ... + \alpha_nx_n^2$ in a suitable basis of $V$, \red{$\alpha_i\in F$}.
We can assume $\alpha_1\neq 0$ since $\q_1$ is non-trivial.
Then $v=(1,0,...,0)$ satisfies $\q_1(v)=\alpha_1\neq 0$.
Assume $r \geq 2$. 
By induction hypothesis, we can pick $v_1 \in V$ such that $\q_i(v_1) \neq 0$ for $ i =1,2,...,r-1$.
If $\q_r(v_1) \neq 0$ then we are done. 
Otherwise, pick $v_2 \in V$ such that $\q_r(v_2) \neq 0$, and
choose $\eps \in F$ such that $\eps^2$ is not in the set
\begin{equation*}
 \left\{\frac{\q_i(v_2)}{\q_i(v_1)} \mid 1 \leq i \leq r-1\right\}
\end{equation*} 
of cardinality at most $r-1$.
Note that such an $\eps$ exists because the Frobenius morphism $F \rightarrow F, u \mapsto u^2$ is injective, and hence the set $\{\eps^2 \mid \eps \in F\}$ contains $|F|\geq r$ many elements. 
Then the vector $v=\eps v_1+v_2$ satisfies $\q_i(v)\neq 0$ for $i=1,...,r$ since 
$$\q_i(\eps v_1 + v_2) = \q_i(\eps v_1) + \q_i(v_2) = \eps^2 \q_i(v_1) + \q_i(v_2) \neq 0\hspace{3ex}\text{for}\ \ i=1,...,r-1,$$ and  $\q_r(\eps v_1 + v_2) = \eps^2 \q_r(v_1) + \q_r(v_2) = \q_r(v_2) \neq 0$. 
\end{proof}

In order to prove Theorem \ref{thm:ChainLemma} for finite fields of characteristic $2$ other than $\F_2$ we need the following lemma.

\begin{lemma}
\label{lem:eChaineHat}
Let $F\neq \F_2$ be a finite field of characteristic 2, and let $n\geq 4$ be an even integer.
Assume that any two orthogonal bases of an inner product space over $F$ of dimension smaller than $n$ are chain equivalent.
Then the standard orthonormal bases $e$ and the orthogonal basis $\hat{e}$ of $\langle 1,1,...,1\rangle = \langle 1 \rangle ^{\oplus n}$ below are chain equivalent:
$$e=(e_1,e_2,...,e_n) \approx \hat{e}= (\hat{e}_1, \hat{e}_2,...,\hat{e}_n)$$
where $\hat{e}_r = \sum_{1 \leq i \neq r \leq n }e_i$.
\end{lemma}

\begin{proof}
The orthogonal basis $e=(e_1,e_2,...,e_n)$ is chain equivalent to an orthogonal basis $u=(u_1,...,u_n)$ with $u_1=a_1e_1+ \cdots +a_ne_n$ if  for $r=1,...,n$ we have $\sum_{1 \leq i \leq r }a_i \neq 0$;
see Lemma \ref{lem:ElementChains}.
Similarly, $\hat{e}=(\hat{e}_1, \hat{e}_2,...,\hat{e}_n)$ is chain equivalent to an orthogonal basis $v=(v_1,...,v_n)$ with $v_1=b_1\hat{e}_1+ \cdots +b_n\hat{e}_n$ if  for $r=1,...,n$ we have $\sum_{1 \leq i \leq r }b_i \neq 0$.
Note that
$$v_1=b_1\hat{e}_1+ \cdots + b_n\hat{e}_n = \hat{b}_1 e_1 + \cdots + \hat{b}_ne_n$$
where $\hat{b}_r = \sum_{1 \leq i \neq r \leq n }b_i$.
Choose elements $b_1,b_n\in F$ such that $b_1,b_n,b_1+b_n\neq 0$.
This is possible since $F$ has more than $2$ elements.
Set $b_i=0$ for $1<i<n$ and $a_i=\hat{b}_i$.
Then 
$$\hat{b}_i= \left\{\begin{array}{cl}b_n & i=1\\ b_1+b_n & 1<i<n \\ b_1 & i=n \end{array}\right.$$
and therefore, for $r=1,...,n$, we have 
$$\sum_{1 \leq i \leq r }a_i = \sum_{1 \leq i \leq r }\hat{b}_i  = 
\left\{\begin{array}{cl} b_n & 1 \leq r <n,\ r\ \text{odd}\\ b_1 & 1 \leq r <n,\ r\ \text{even}\\ b_1+b_n & r=n,
\end{array}\right.
$$
and 
$$\sum_{1 \leq i \leq r }b_i  = 
\left\{\begin{array}{cl} b_1 & 1 \leq r <n\\ b_1+b_n &  r=n.
\end{array}\right.
$$
In particular, the last two sums are non-zero for $r=1,...,n$.
Hence, there are orthogonal bases $u$ and $v$ as above with $e \approx u$, $\hat{e}\approx v$ and $u_1=v_1$.
By assumption applied to the inner product space $u_1^{\perp}=v_1^{\perp}$ of dimension $n-1$, we have $(u_2,...,u_n) \approx (v_2,...,v_n)$.
Therefore, 
$$e \approx u \approx v \approx \hat{e}.$$
\end{proof}

\begin{example}
As an illustration of Lemma \ref{lem:eChaineHat}, the following explicitly shows that $(e_1,e_2,e_3,e_4) \approx (\hat{e}_1, \hat{e}_2,\hat{e}_3,\hat{e}_4) \ \red{\in \langle 1,1,1,1\rangle}$ over $\F_4=\F_2[\alpha]/(\alpha^2+\alpha+1)$ where we set $\beta=1+\alpha$ and note that  $\alpha\beta=1$, $\alpha+\beta=1$, $\alpha^2=\beta$, $\beta^2=\alpha$:
{\tiny
$$\renewcommand\arraystretch{1.2}
\begin{array}{rlccr}
 & (e_1,&e_2,&e_3,&e_4)\\
 \approx & (\alpha e_1+\beta e_2,&\beta e_1+ \alpha e_2,&e_3,&e_4)\\
 \approx & ( e_1+\alpha e_2+\alpha e_3,&\beta e_1+ \alpha e_2,&\beta e_1 +  e_2+\beta e_3,&e_4)\\
  \approx & (\beta e_1+ e_2+ e_3+\alpha e_4,& \beta e_1+ \alpha e_2 ,&  \beta e_1 +  e_2+\beta e_3,&\alpha  e_1 + \beta e_2 + \beta e_3 + \beta e_4)\\
 \approx & (\beta e_1+ e_2+ e_3+\alpha e_4,&\beta e_1+ \alpha e_2,&  e_1 +  \alpha e_2+\beta e_3+ e_4 ,&  \beta  e_3 + \alpha e_4)\\
 \approx & (\beta e_1+ e_2+ e_3+\alpha e_4,&e_1 +  \beta e_2+\alpha e_3+ e_4,&  e_1 +  \alpha e_2+\beta e_3+ e_4 ,&\alpha  e_1 + e_2 + e_3 + \beta e_4)\\
 \approx & (\beta e_1+ e_2+ e_3+\alpha e_4,& e_1+ e_3 +  e_4,&  e_1 +  e_2+e_4,&\alpha  e_1 +  e_2 +  e_3 + \beta e_4)\\
 \approx & ( e_2+ e_3+ e_4,& e_1+ e_3 +  e_4,&  e_1 +  e_2+e_4,&  e_1 +  e_2 +  e_3 )\\
= & (\hat{e}_1,&\hat{e}_2,&\hat{e}_3,&\hat{e}_4)
\end{array}
$$
}
In contrast, over $\F_2$ we have  $(e_1,e_2,e_3,e_4) \not\approx (\hat{e}_1, \hat{e}_2,\hat{e}_3,\hat{e}_4)$; see Remark \ref{rmk:NoF2Chain}.
\end{example}

\begin{theorem}
\label{thm:FiniteFieldChainLem}
Let $F$ be a finite field of characteristic 2 such that $F\neq \F_2$.
Let $V$ be an inner product space over $F$.
Then any two orthogonal bases of $V$ are chain equivalent.
\end{theorem}

\begin{proof}
We proceed by induction on the dimension $n =\dim_FV$ of $V$.
For $n=0,1,2$, there is nothing to prove, and the case $n=3$ was treated in Theorem \ref{thm:FieldChainLem}.
Thus, we can assume $n\geq 4$.
Let $v=(v_1,v_2,...,v_n)$ and  $w=(w_1,w_2,w_3,...,w_n)$ be two orthogonal bases of $V$.
Among all orthogonal bases of $V$ that are chain equivalent to $v$ choose one, say $u=(u_1,u_2,u_3,...,u_n)$,
such that for the linear combination $w_1 = a_1u_1 + \cdots + a_nu_n$ the number $r$ of non-zero coefficients $a_i\neq 0$ is minimal. 
Reordering, we can assume $a_1,...,a_r \neq 0$ and $a_{r+1}=\cdots =a_n=0$.
Clearly $1 \leq r \leq n$.
If $r=1$ then $v \approx u \approx (w_1,u_2,u_3,...,u_n) \approx (w_1,w_2,...,w_n)$ since $(u_2,u_3,...u_n) \approx (w_2,...,w_n)$, by induction hypothesis applied to the orthogonal complement $w_1^{\perp}$ of $w_1$ inside $V$.
If $r = 2$ then 
$v \approx u \approx (w_1,u'_2,u_3,...,u_n)$
where $u'_2$ is a non-zero vector of the orthogonal complement of $w_1$ inside of $Fu_1 \perp Fu_2$.
Then $v \approx (w_1,u'_2,u_3,...,u_n) \approx (w_1,w_2,...,w_n)$ since $(u'_2,u_3,...u_n) \approx (w_2,...,w_n)$, by induction hypothesis applied to the orthogonal complement $w_1^{\perp}$ of $w_1$ inside $V$.
Assume $r\geq 3$.
Since every element in $F$ is a square, we can rescale and assume $\q(u_i)=\q(w_i)=1$, $i=1,...,n$ as rescaling yields chain equivalent bases.
Assume that there is a pair $1 \leq i\neq j\leq r$ such that $a_iu_i+a_ju_j$ is anisotropic.
After reordering, we can assume $i=1$, $j=2$.
Set $u_1'=a_1u_1+a_2u_2$, and let $u_2'$ be a non-zero vector in the orthogonal complement of $u_1'$ inside $Fu_1\perp Fu_2$.
Then $u \approx (u_1',u_2',u_3,...u_n)$ and $w_1=u'_1 + a_3u_3 + \cdots +a_ru_r$ contradicting minimality of $r$.
Thus, for all pairs $1\leq i,j \leq r$, the vector $a_iu_i+a_ju_j$ is isotropic, that is, $0 = \q(a_iu_i+a_ju_j)=a_i^2\q(u_i) + a_j^2\q(u_j) = a_i^2+a_j^2=(a_i+a_j)^2$, so $a_i+a_j=0$, for $1\leq i\leq r$, that is,  $a=a_1=a_2=a_3=\cdots =a_r \neq 0$.
Then $w_1=a(u_1+ \cdots + u_r)$.
Since $1= \q(w_1)=a^2 (\q(u_1) + \cdots +\q(u_r))=ra^2$, the positive integer $r$ is odd.
Therefore, $1=ra^2=a^2$ implies $a=1$, and we have
$w_1=u_1+ \cdots +u_r$.
If $r<n$, we can use Lemma \ref{lem:eChaineHat} to find an orthogonal basis $u_2',...,u'_{r+1}$ of $Fu_2 \perp ...\perp Fu_{r+1}$ such that $(u_1,...,u_{r+1}) \approx (w_1,u'_2,...,u'_{r+1})$.
Then
$$v \approx (u_1,u_2,u_3,...,u_n) \approx (w_1,u'_2,u_3,...,u'_{r+1},u_{r+2},...,u_n) \approx w$$
since $(u'_2,u_3,...,u'_{r+1},u_{r+2},...,u_n) \approx (w_2,w_3,...,w_n)$, by the induction hypothesis applied to $w_1^{\perp}$ inside $V$.
Finally, the case $r=n$ is impossible.
Indeed, if $r=n$, then every vector in $w_1^{\perp}\subset V$ is isotropic contradicting the the assumption that $(w_2,...,w_n)$ is an orthogonal basis of $w_1^{\perp}$.
\end{proof}

\begin{proof}[Proof of Theorem \ref{thm:ChainLemma}]
The analog of Theorem \ref{thm:FieldChainLem} for fields $F$ of characteristic not $2$ is classical \red{\cite[Satz 7]{Witt}} and holds without restriction on the size of $F$; see for instance \cite[Theorem I.5.2]{Lam:book}. 
Together with Theorems \ref{thm:FieldChainLem} and \ref{thm:FiniteFieldChainLem}, this implies Theorem \ref{thm:ChainLemma} in view of Lemma \ref{lem:ChainReductionRtoF}.
\end{proof}

\begin{remark}
\label{rmk:NoF2Chain}
The Chain Lemma does not hold for $R=F=\F_2$ and $V=\F_2^4$ equipped with the form \red{$\langle 1,1,1,1 \rangle$}.
The orthogonal basis $e = \{e_1,e_2,e_3,e_4\}$ is only chain equivalent to itself since $\langle 1 \rangle \perp \langle 1 \rangle$ has unique orthogonal basis $\{e_1,e_2\}$.
But $V$ has also orthogonal basis $\hat{e}=\{\hat{e}_1,\hat{e}_2,\hat{e}_3,\hat{e}_4\}$ where $\hat{e}_i = e_1+e_2+e_3 +e_4 -e_i$ for $i=1,...,4$.
In particular, the two orthogonal basis $e$ and $\hat{e}$ of $V$ are not chain equivalent.
\end{remark}

\section{Presentation of $GW(R)$}
\label{sec:Presentation}

For an invertible symmetric matrix $A\in M_n(R)$, we denote by $\langle A \rangle$ the inner product space $R^n$ equipped with the form $\red{\b( x,y)} ={^tx}Ay$, $x,y\in R^n$ where ${^tx}$ denotes the transpose of the column vector $x$.
The following shows that every inner product space stably admits an orthogonal basis. 
In particular, the ring homomorphism (\ref{eqn:ZRGWringhomo}) is surjective.

\begin{proposition}
\label{prop:GWgens}
Let $(R,\m,F)$ be a commutative local ring.
\begin{enumerate}
\item
\label{prop:GWgens:item1}
For any inner product space \red{$V$} over $R$ there is an isometry 
$$\red{V} \cong \left\langle u_{1}\right\rangle \perp\dots\perp\left\langle u_{l}\right\rangle \perp N_{1}\perp\dots\perp N_{r}$$
for some $u_{i}\in R^{*}$ and $N_{i}=\left\langle \left(\begin{smallmatrix}a_{i} & 1\\
1 & b_{i}
\end{smallmatrix}\right)\right\rangle$ with $a_{i},b_{i}\in \m$.
\item
\label{prop:GWgens:item2}
For any $a,b\in \m$ there is an isometry of inner product spaces
$$\left\langle \left( \begin{matrix}a & 1\\
1 & b
\end{matrix}\right) \right\rangle + \langle -1 \rangle \cong \left\langle \frac{1-ab}{(-1+a)(-1+b)}\right\rangle + \langle -1+a \rangle + \langle -1+b\rangle.$$
\item
\label{prop:GWgens:item3}
For any inner product space $V$ over $R$, there is an inner product space $W$ with orthogonal basis such that $V\perp W$ has an orthogonal basis.
In particular, the Grothendieck-Witt group $GW(R)$ of inner product spaces is additively generated by one-dimensional spaces $\langle u \rangle$, $u \in R^*$.
\end{enumerate}
\end{proposition}

\begin{proof}
For part (\ref{prop:GWgens:item1}),
if $\red{\q(x)=\b(x, x)} = u \in R^*$ is a unit for some $x\in \red{V}$ then $\red{V} = Rx \perp (Rx)^{\perp}$ is a decomposition into non-degenerate subspaces, and $Rx = \langle u \rangle$.
Hence, repeatedly splitting off one-dimensional inner product spaces, we can write 
$\red{V}=\left\langle u_{1}\right\rangle \perp\dots\perp\left\langle u_{l}\right\rangle \perp N$
where $u_{i}\in R^{*}$ and $\q(x)\in \m$ for all $x\in N$. 
If $N\neq 0$ then the rank of $N$ is at least $2$, and we can find $x,y \in N$ such that $\b(x, y) = 1$.
The subspace  $N_1$ spanned by $x$ and $y$ is non-degenerate with Gram matrix 
$\left( \begin{smallmatrix}a & 1\\
1 & b
\end{smallmatrix}\right)$
where $a = \q(x)$ and $b = \q(y)$.
In particular, $N = N_1\perp N_1^{\perp}$ is a decomposition into non-degenerate subspaces, and
$N_1 = \left\langle \left(\begin{smallmatrix}a & 1\\
1 & b
\end{smallmatrix}\right)\right\rangle$.
Now we keep splitting off rank $2$ spaces $N_i$ to obtain the desired form.

Part (\ref{prop:GWgens:item2}) follows from the equation \red{in $M_3(R)$}
$$
\begin{array}{l}
\left( \begin{matrix} 
-\frac{1}{-1+a} & -\frac{1}{-1+b}  & \frac{-1+ab}{(-1+a)(-1+b)}\\
-1 & 0 & 1\\
0 & -1 & 1
\end{matrix}\right)
\left( \begin{matrix}a & 1 & 0 \\
1 & b & 0 \\ 0 & 0 & -1
\end{matrix}\right) 
\left( \begin{matrix} -\frac{1}{-1+a}  & -1 & 0 \\
-\frac{1}{-1+b} & 0 & -1 \\ \frac{-1+ab}{(-1+a)(-1+b)} & 1& 1
\end{matrix}\right)\\
\\
= 
\left( \begin{matrix}
\frac{1-ab}{(-1+a)(-1+b)} & 0 & 0\\
0 & -1+a & 0\\
0 & 0 & -1+b
\end{matrix}\right).
\end{array}$$
Finally, (\ref{prop:GWgens:item3}) follows from (\ref{prop:GWgens:item1}) and (\ref{prop:GWgens:item2}).
\end{proof}

\begin{lemma}
\label{lem:kerpiGens}
  Let $(R,\m,F)$ be a commutative local ring with residue field $F\neq \F_2$.
  Then the kernel $\ker(\pi)$ of the ring homomorphism (\ref{eqn:ZRGWringhomo}) is generated as abelian subgroup of $\mathbb{Z}[R^*]$ by the following elements: 
  \begin{center}
$\langle \alpha \rangle - \langle \beta \rangle$ with $\alpha,\beta\in R^{*}$ and $\langle \alpha \rangle \cong  \langle \beta \rangle $
\vspace{1ex}

$\langle \alpha \rangle + \langle \beta \rangle - \langle \gamma \rangle - \langle \delta \rangle$ with $\alpha,\beta,\gamma,\delta\in R^{*}$ and $\langle \alpha,\beta \rangle \cong  \langle \gamma,\delta \rangle $.
\end{center}
  
\end{lemma}

\begin{proof}
By definition, an element $\sum_{i=1}^n \langle a_{i} \rangle - \sum_{j=1}^m \langle b_{j} \rangle$ of $\Z[R^*]$ with $a_{i},b_{j} \in R^*$ is in 
 $\ker(\pi) \subset \Z[R^*]$ if and only if there is an inner product space $K$ and an isometry \red{of inner product spaces}
\begin{equation}
\label{eqn:KoplusDiag}
\red{\langle a_{1},...,a_{n} \rangle \oplus K  \cong \langle b_{1},...,b_{m} \rangle \oplus K}.
\end{equation}
In particular, $n=m$.
By Proposition \ref{prop:GWgens} (\ref{prop:GWgens:item2}), there exists an inner product space $W$ over $R$ such that $K \oplus W$ admits an orthogonal basis.
Replacing $K$ with $K\oplus W$, we can assume that $K$ in (\ref{eqn:KoplusDiag}) has an orthogonal basis, say $\{z_1 ,...,z_l\}$. 
The inner product space $\red{(V,\b)} \defeq \langle a_{1},...,a_{n} \rangle \oplus K  \cong \langle b_{1},...,b_{n} \rangle \oplus K$ has the following two orthogonal bases:
$$
A =\{x_{1},...,x_{n},z_{1},..,z_{l}\}, \text{with } \b(x_{i},x_{i}) = a_{i}, \text{ and } \b(z_{i},z_{i}) = c_{i},\ \text{and}$$
$$
B = \{y_{1},...,y_{n},z\red{'}_{1},..,z\red{'}_{l}\}, \text{with } \b(y_{i},y_{i}) = b_{i}, \text{ and } \b(z\red{'}_{i},z\red{'}_{i}) = c_{i}.\phantom{andm}
$$
By Theorem \ref{thm:ChainLemma}, we can choose a chain of orthogonal bases, $C_{0}, C_{1} , ... , C_{N-1} , C_{N}$ such that $C_i$ and $C_{i+1}$ differ in at most $2$ elements, $i=0,...,N-1$, and $C_0=A$, $C_N=B$.
Let $\langle c_{1}^{(i)},...,c_{n+l}^{(i)} \rangle$ be the diagonal form corresponding to $C_{i}$. As $C_{i}$ and $C_{i+1}$ differ in at most two vectors, 
\begin{center}
\red{$(\langle c_{1}^{(i)}\rangle +...+\langle c_{n+l}^{(i)}\rangle ) - (\langle c_{1}^{(i+1)}\rangle +...+\langle c_{n+l}^{(i+1)})\rangle \in \Z[R^*]$} 
\end{center}
is of the form 
\begin{center}
\red{$\langle a\rangle  - \langle b\rangle \in \Z[R^*]$} with $\langle a \rangle \cong  \langle b \rangle $ \\
or \\
\red{$\langle a\rangle  + \langle b\rangle - \langle a'\rangle - \langle b'\rangle \in \Z[R^*]$} with $\langle a,b \rangle \cong  \langle a',b' \rangle $.
\end{center}
\red{
In $\Z[R^*]$, we have 
$$\renewcommand\arraystretch{2}
\begin{array}{rcl}
\sum_{i=1}^{n}\langle a_{i}\rangle - \sum_{j=1}^{n}\langle b_{j}\rangle &=& (\sum_{i=1}^{n}\langle a_{i}\rangle + \sum_{i=1}^{l}\langle c_{i}\rangle ) - (\sum_{j=1}^{n}\langle b_{j}\rangle + \sum_{i=1}^{l}\langle c_{i}\rangle)\\
&=& \sum_{i=1}^{n+l}\langle c_{i}^{(0)}\rangle - \sum_{j=1}^{n+l}\langle c_{j}^{(N)}\rangle \\
&=& \sum_{k=0}^{N-1}(\sum_{i=0}^{n+l}\langle c_{i}^{(k)}\rangle - \sum_{i=0}^{n+l}\langle c_{i}^{(k+1)}\rangle),
\end{array}
$$
which is of the desired form}.
\end{proof}

\begin{lemma}
\label{lem:freps}
Let $R$ be a commutative ring.
Assume we have an isometry of inner product spaces $\langle a,b\rangle \cong \langle c,d\rangle$ over $R$ where $a,b,c,d\in R^*$ with $d=abc$ and $c = ax^2+by^2$, $x,y\in R$.
If \red{in $R$, we have} $f=as^2 +bt^2$, then \red{the following equation holds in $R$}
$$f=c \left(\frac{asx +bty}{c}\right)^2 + d\left(\frac{tx - sy}{c}\right)^2.$$
\end{lemma}

\begin{proof}
Direct verification.
\end{proof}

For a commutative local ring $R$, let $K^{MW}_0(R)$ be the quotient ring of $\Z[R^*]$ modulo the ideal generated by the relations (\ref{thm:presentation:SquareTriviality}), (\ref{thm:presentation:ScalarInv}) and (\ref{thm:presentation:SteinbergRln}) of Theorem \ref{thm:presentation} where $\mylangle a \myrangle = 1- \langle a \rangle$, and $\langle a \rangle \in \Z[R^*]$ is the element corresponding to $a\in R^*$.

\begin{lemma}
\label{lem:GWtoKMWrlns2}
Let $(R,\m,F)$ be a commutative local ring with residue field $F\neq \F_2$, and let $a,b,c,d \in R^*$ with $\langle a,b\rangle \cong \langle c,d \rangle$ as inner product spaces over $R$.
Then the following equality holds in $K_0^{MW}(R)$:
$$\langle a\rangle +\langle b\rangle =\langle c\rangle + \langle d\rangle.$$
\end{lemma}

\begin{proof}
The isometry $\langle a,b\rangle \cong \langle c,d\rangle$ implies 
$c = ax^2+by^2 \in R$ for some $x,y\in R$ and $d=abc\in \red{R^*/(R^*)^2}$.
Since $\langle r^2d\rangle = \langle d\rangle \in K_0^{MW}(R)$, we can assume $d=abc\in R^*$.
If $x,y\in R^*$, we say that $c$ is {\em regularly represented} by $\langle a,b\rangle$.
In this case 
$$\renewcommand\arraystretch{1.5}
\begin{array}{rcl}
\langle a\rangle + \langle b\rangle &=& \langle ax^2\rangle + \langle by^2\rangle \\
&=& \langle c\rangle \left(\langle ac^{-1}x^2\rangle + \langle bc^{-1}y^2\rangle\right)\\
&=&  \langle c\rangle \left(\langle 1\rangle + \langle abc^{-2}x^2y^2\rangle \right)\\
&=& \langle c\rangle + \langle d\rangle
\end{array}
$$
in $K_0^{MW}(R)$ where we used the Steinberg relation for the third equality.

Assume now that one of $x$ or $y$ is in the maximal ideal $\m$ of $R$, then the other is a unit since $c$ is a unit. 
Without loss of generality, we can assume $x\in R^*$ and $y\in \m$.
We claim that if there is $z\in R^*$ such that $ax^2+bz^2\in R^*$, then
$\langle a\rangle +\langle b\rangle =\langle c\rangle + \langle d\rangle \in K^{MW}_0(R).$
Indeed, given $z\in R^*$ such that $\gamma=ax^2+bz^2\in R^*$ we set $\delta=ab\gamma$.
Then $\langle a,b\rangle \cong \langle \gamma,\delta\rangle$,  and $\gamma$ is regularly represented by $\langle a,b\rangle$.
In particular, $\langle \gamma \rangle + \langle \delta\rangle = \langle a \rangle + \langle b\rangle \in K_0^{MW}(R)$.
Since $c=ax^2+by^2$, Lemma \ref{lem:freps} yields
$$c=\gamma \left(\frac{ax^2 +byz}{\gamma}\right)^2 + \delta\left(\frac{xy - xz}{\gamma}\right)^2.$$
Note that $(ax^2 +byz)\gamma^{-1}$ and $(xy - xz)\gamma^{-1}$ are units in $R$ since $x,z,a,b,\gamma \in R^*$ and $y\in \m$.
In particular, $c$ is regularly represented by $\langle \gamma,\delta\rangle$ and thus $\langle c \rangle + \langle d\rangle = \langle \gamma \rangle + \langle \delta\rangle \in K_0^{MW}(R)$.
Hence,
$$\langle c \rangle + \langle d\rangle = \langle \gamma \rangle + \langle \delta\rangle = \langle a \rangle + \langle b\rangle\hspace{2ex}\in\hspace{2ex}K_0^{MW}(R).$$

If $F\neq \F_3$ (and $F\neq \F_2$, by assumption) then we can find an element $z\in R^*$ with $ax^2+bz^2\in R^*$  as in this case $F$ has at least \red{$2$ square} units, and we only need to make sure that its class $\bar{z}$ in $F=R/\m$ satisfies $\bar{z}^2 \neq\red{  -\bar{a}\bar{b}^{-1}\bar{x}^2\in F}$.
If  there is no $z\in R^*$ such that $ax^2+bz^2\in R^*$, then $F=\F_3$ and
$a+b, a-c\in \m$ as in this case square units in $R$ are $1$ modulo $\m$.
Then $\langle c,-b\rangle \cong \langle a, -d\rangle$ since $a = c(1/x)^2 -b(y/x)^2 $ and $d=abc$.
Note that there is $z\in R^*$ such that $\gamma = c (1/x)^2  -bz^2 \in R^*$.
For instance, $z=1/x \in R^*$ will do since $c-b=2 c -(a+b)+(a-c)\in R^*$.
As proved above, this implies $\langle c \rangle + \langle -b\rangle = \langle a\rangle + \langle -d\rangle$ in $K_0^{MW}(R)$.
Using relation (\ref{thm:presentation:ScalarInv}) of Theorem (\ref{thm:presentation}) which holds in $K^{MW}_0(R)$, we have
$$\langle a\rangle + \langle b \rangle = \langle a \rangle - \langle -b \rangle + h = \langle c \rangle - \langle -d \rangle + h = \langle c \rangle + \langle d \rangle\hspace{2ex}\in\hspace{2ex}K_0^{MW}(R).$$
\end{proof}    

\begin{corollary}
\label{cor:KMWGW}
Let $(R,\m,F)$ be a commutative local ring with residue field $F\neq \F_2$.
Then the surjection (\ref{eqn:ZRGWringhomo}) induces an isomorphism $$K_0^{MW}(R) \stackrel{\cong}{\longrightarrow} GW(R).$$
\end{corollary}

\begin{proof}
Let $J \subset \mathbb{Z}[R^*]$ be the ideal generated by the relations  (\ref{thm:presentation:SquareTriviality}), (\ref{thm:presentation:ScalarInv}) and (\ref{thm:presentation:SteinbergRln}) of Theorem \ref{thm:presentation}, that is, $J$ is the kernel of the ring homomorphism $\Z[R^*]\to K^{MW}_0(R)$.
As before, let $\pi : \mathbb{Z}[R^*] \rightarrow GW(R)$, $\langle a \rangle \mapsto \langle a \rangle$ be the canonical ring homomorphism (\ref{eqn:ZRGWringhomo}).
It is well known that $J \subset \ker \pi$. 
Indeed, the first relation is the isometry $\langle u\rangle \cong \langle a^2u\rangle$ given by the multiplication with $a\in R^*$, the second relation follows from \red{the equation in $M_2(R)$}
$$\begin{pmatrix}
0 & 1 \\
u & 0
\end{pmatrix} 
\begin{pmatrix}
0 & 1 \\
1 & 0 
\end{pmatrix}
\begin{pmatrix}
0& u \\
1 & 0 
\end{pmatrix}
=
\begin{pmatrix}
0 & u \\
u& 0
\end{pmatrix},$$
that is, $\left\langle \left(\begin{smallmatrix}0&1\\ 1&0\end{smallmatrix}\right)\right\rangle \cong \langle u \rangle \cdot \left\langle \left( \begin{smallmatrix}0&1\\ 1&0\end{smallmatrix}\right)\right\rangle$, and the equality $\left\langle \left(\begin{smallmatrix}0&1\\ 1&0\end{smallmatrix}\right)\right\rangle = h \in GW(R)$ in view of Proposition \ref{prop:GWgens} (\ref{prop:GWgens:item2}) with $a=b=0$.
The last relation is a consequence of the equality \red{in $M_2(R)$}
$$\begin{pmatrix}
1 & -1 \\
1-a & a 
\end{pmatrix} 
\begin{pmatrix}
a & 0 \\
0 & 1-a 
\end{pmatrix}
\begin{pmatrix}
1 & 1-a \\
-1 & a 
\end{pmatrix}
=
\begin{pmatrix}
1 & 0 \\
0& a(1-a) 
\end{pmatrix}.$$
Lemma \ref{lem:kerpiGens} gives us additive generators of $\ker(\pi)$.
By definition of $K^{MW}_0(R)$ and Lemma \ref{lem:GWtoKMWrlns2}, these generators are in $J$, and so, $J=\ker(\pi)$.
\end{proof}

We finish the section with a proof of Remark \ref{rmk:IrrelvRlns}.
Let $\tilde{K}_0^{MW}(R)$ be the ring quotient of $\Z[R^*]$ modulo the Steinberg relation (\ref{thm:presentation:SteinbergRln}) of Theorem \ref{thm:presentation}.

\begin{lemma}
\label{lem:GWtoKMWrlns}
Let $(R,\m,F)$ be a commutative local ring with residue field $F\neq \F_2,\F_3$.
Then for all $a\in R^*$, the following holds in $\tilde{K}_0^{MW}(R)$:
\begin{enumerate}
\item 
\label{lem:GWtoKMWrlns:item1}
$\mylangle a \myrangle \mylangle -a \myrangle=0$,
\item 
\label{lem:GWtoKMWrlns:item2}
$\mylangle a^2\myrangle = \mylangle a \myrangle \cdot h$.
\end{enumerate}
\end{lemma}

\begin{proof}
Part (\ref{lem:GWtoKMWrlns:item1}) was implicitly proved in \cite[Lemma 4.4]{myEuler}.
The analogous arguments for Milnor $K$-theory are due to \cite{Milnor:KthQuadF}.
We give the relevant details here.
First assume $\bar{a}\neq 1$ where $\bar{a}$ means reduction modulo the maximal ideal $\m\subset R$.
Then $1-a, 1-a^{-1}\in R^*$.
Therefore, in \red{$\tilde{K}_0^{MW}(R)$}, we have
$$\renewcommand\arraystretch{1.5}
\begin{array}{rcl}
\mylangle a\myrangle \mylangle-a\myrangle &= & \mylangle a\myrangle \left(\mylangle 1-a\myrangle-\langle -a\rangle \mylangle 1-a^{-1}\myrangle\right)\\
& = &
-\langle -a\rangle \mylangle a\myrangle \mylangle1-a^{-1}\myrangle = \langle -a\rangle \langle a\rangle \mylangle a^{-1}\myrangle \mylangle 1-a^{-1}\myrangle\\
& =&0.
\end{array}$$
If $\bar{a}=1$, choose $b\in R^*$ with $\bar{b}\neq 1$.
This is possible since $F \neq \F_2$.
Then $\bar{a}\bar{b}\neq 1$.
Therefore, in \red{$\tilde{K}_0^{MW}(R)$}, we have
$$
\renewcommand\arraystretch{1.5}
\begin{array}{rcl}
0 &= &\mylangle ab \myrangle \mylangle -ab \myrangle = 
 \mylangle a \myrangle \left(\mylangle -a \myrangle +\langle -a\rangle \mylangle b \myrangle \right)+\langle a\rangle \mylangle b \myrangle \left(\mylangle a \myrangle +\langle a\rangle\mylangle -b \myrangle \right)\\
&=&\mylangle a \myrangle \mylangle -a \myrangle + h\langle a \rangle \mylangle a \myrangle \mylangle b \myrangle.  
\end{array}$$
Hence, for all $\bar{b}\neq 1$ we have
$\mylangle a \myrangle \mylangle -a \myrangle =-h\langle a \rangle \mylangle a \myrangle \mylangle b \myrangle \ \red{\in \tilde{K}_0^{MW}(R)}$. 
Now, choose $b_1,b_2\in A^*$ such that $\bar{b}_1,\bar{b}_2,\bar{b}_1\bar{b}_2\neq 1$. 
This is possible since $|F|\geq 4$.
Then \red{in $\tilde{K}_0^{MW}(R)$ we have}
$$\renewcommand\arraystretch{1.5}
\begin{array}{rcl}
\mylangle a \myrangle \mylangle -a \myrangle& = &-h\langle a \rangle \mylangle a \myrangle \mylangle b_1b_2 \myrangle\\
&=&-h\langle a \rangle \mylangle a \myrangle ( \mylangle b_1\myrangle + \langle b_1\rangle \mylangle b_2\myrangle)\\
&=& \mylangle a \myrangle \mylangle -a \myrangle +\langle b_1 \rangle \mylangle a \myrangle \mylangle -a \myrangle .
\end{array}$$
Hence, $\langle b_1\rangle \mylangle a \myrangle \mylangle -a \myrangle =0\ \red{\in \tilde{K}_0^{MW}(R)}$.
Multiplying with $\langle b_1^{-1}\rangle $ yields the result.

In $\Z[R^*]$ we have
$\mylangle a \myrangle \mylangle -a \myrangle \cdot \langle -1 \rangle + \mylangle a^2 \myrangle  = \mylangle a \myrangle\cdot h$ which implies part (\ref{lem:GWtoKMWrlns:item2}).
\end{proof}

\section{An example of $GW(R)\ncong K_0^{MW}(R)$}

\red{For any commutative local ring $R$, the three defining relations for $K^{MW}_0(R)$ hold in $GW(R)$; see the proof of Corollary \ref{cor:KMWGW}.
In particular, the map \red{(\ref{eqn:ZRGWringhomo}) factors through the quotient $K^{MW}_0(R)$ of $\Z[R^*]$ and induces the ring homomorphism $K^{MW}_0(R) \to GW(R)$ sending the generator $\langle a \rangle$ of $K^{MW}_0(R)$ to the Grothendieck-Witt class of the inner product space $\langle a \rangle$ for $a\in R^*$.
This ring homomorphism} is surjective for any local ring $R$, by Proposition \ref{prop:GWgens}.
Thus, we obtain natural surjective ring homomorphisms
\begin{equation}
\label{eqn1:prop:tildeInotI}
\Z[R^*] \twoheadrightarrow \Z[\red{R^*/(R^*)^2}] \twoheadrightarrow K^{MW}_0(R)\twoheadrightarrow GW(R) \stackrel{\operatorname{rk}}{\twoheadrightarrow} \Z
\end{equation}
where the last map sends an inner product space $(V,\b)$ to the rank $n=\rk(V)$ of the free $R$-module $V\cong R^n$.}

\begin{proposition}
\label{prop:tildeInotI}
For $R = \mathbb{F}_2[x]/(x^4)$, the natural surjection $K_0^{MW}(R) \to GW(R)$ \red{in (\ref{eqn1:prop:tildeInotI})} has kernel $\Z/2$.
In fact, we have isomorphisms of abelian groups 
$$GW(R)\cong \Z \oplus (\Z/2)^2\hspace{2ex}\text{and}\hspace{2ex}K_0^{MW}(R) \cong \Z \oplus (\Z/2)^3.$$
\end{proposition}

\begin{proof}
\red{Let $I_{\Z}\subset \Z[\red{R^*/(R^*)^2}]$, $I_{MW}\subset K^{MW}_0(R)$ and $I \subset GW(R)$ be the respective augmentation ideals, that is, the kernel of the surjective ring homomorphisms (\ref{eqn1:prop:tildeInotI}) from $\Z[\red{R^*/(R^*)^2}]$, $K^{MW}_0(R)$, $GW(R)$ to $\Z$.
The maps (\ref{eqn1:prop:tildeInotI}) induce surjections on augmentation ideals $I_{\Z} \twoheadrightarrow I_{MW} \twoheadrightarrow I$.}
The first part of the proposition is the statement that the surjection $\red{I_{MW}} \twoheadrightarrow I$ has kernel $\Z/2$.

For \red{the local ring} $R = \mathbb{F}_2[x]/(x^4)$, the group \red{of units} $R^*$ has \red{order} $8$ \red{and elements $1+ax+bx^2+cx^3$, where $a,b,c\in \F_2$}.
The group homomorphism $R^* \to R^*: a \mapsto a^2$ has image 
$\{(1+ax+bx^2+cx^3)^2|\ a,b,c\in \F_2\} = \{1,1+x^2\}$.
In particular, the cokernel $\red{R^*/(R^*)^2}$ is a \red{$2$-torsion abelian} group of order $4$. 
Hence, the group $\red{R^*/(R^*)^2}$ is the Klein $4$-group $K_4\red{\cong (\Z/2)^2}$.
A set of coset representatives for $\red{R^*/(R^*)^2}$ is given by the elements
$1$, $1+x$, $1+x+x^2$, $1+x^2+x^3 \in R^*$ since
$(1+x)(1+x^2+x^3)=1+x +x^2 +2x^3+x^4 = 1+x+x^2$ is not a square. 
From the matrix equation \red{in $M_2(R)$}
$$
\begin{pmatrix}
x & 1 \\
1 & x+x^2+x^3 
\end{pmatrix}
\begin{pmatrix}
1 & 0 \\
0 & 1+x 
\end{pmatrix}
\begin{pmatrix}
x & 1 \\
1 & x+x^2+x^3 
\end{pmatrix}=
\begin{pmatrix}
1+x+x^2 & 0 \\
0& 1+x^2+x^3 
\end{pmatrix}$$
we see that 
\begin{equation}
\label{eqn2:prop:tildeInotI}
\langle 1 \rangle + \langle 1+x \rangle = \langle 1+x+x^2 \rangle + \langle 1+x^2+x^3 \rangle \in GW(R).\end{equation}
We have $2I=0$ as $h=\langle 1\rangle + \langle -1\rangle \red{=\langle 1\rangle + \langle 1 \rangle} = 2$, thus $0 = \mylangle u \myrangle h = 2 \mylangle u \myrangle \in I$ for all $u\in R^*$, and $I$ is additively generated by $\mylangle u\myrangle$, $u\in R^*$.
\red{In view of (\ref{eqn2:prop:tildeInotI}) and $2I=0$}, we obtain \red{the equality in $GW(R)$}
\begin{equation}
\label{eqn3:prop:tildeInotI}
0 = \mylangle 1+x\myrangle + \mylangle 1+x+x^2 \myrangle + \mylangle 1+x^2+x^3 \myrangle = \sum_{\red{w \in R^*/(R^*)^2}}\mylangle \red{w} \myrangle
\end{equation}
from which we see that $I^2=0$.
Indeed, for $u \in \red{R^*/(R^*)^2}$ we have $\mylangle u \myrangle ^2 = 2\mylangle u \myrangle =0\ \red{\in GW(R)}$, and for $v\neq u \in \red{R^*/(R^*)^2}$, $u,v\neq 1 \in \red{R^*/(R^*)^2}$, we have \red{from (\ref{eqn3:prop:tildeInotI})}
$$\mylangle u \myrangle\mylangle v \myrangle = \mylangle u \myrangle + \mylangle v \myrangle + \mylangle uv \myrangle = \sum_{w \in \red{R^*/(R^*)^2}}\mylangle w \myrangle = 0 \in GW(R).$$
Recall the isomorphism $\red{R^*/(R^*)^2}\cong I/I^2: a \mapsto \mylangle a \myrangle$ with inverse the map that sends an inner product space $(V,\b)$ to the determinant of the Gram matrix of $\b$.
In our case, this yields $I=I/I^2\ \red{\cong R^*/(R^*)^2} \cong (\Z/2)^2$.

To compute \red{$I_{MW}$} for $R=\mathbb{F}_2[x]/(x^4)$, we note that if $a\in R$ is a unit then $1-a$ is not a unit and the Steinberg relation is vacuous. 
Moreover, $\mylangle u \myrangle h = 2 \mylangle u \myrangle\in \Z[R^*]$ as $h=\red{\langle 1 \rangle + \langle -1\rangle = \langle 1 \rangle + \langle 1 \rangle =}2 \in \red{\Z[R^*]}$, and thus, 
$K_0^{MW}(R)$ is the quotient of $\Z[\red{R^*/(R^*)^2}]$ by the relation $2\mylangle u \myrangle =0$ for $u \in \red{R^*/(R^*)^2}$.
\red{Since $I_{\Z}$ is additively generated by the elements $\mylangle u\myrangle$ for $u\in R^*/(R^*)^2$, we therefore have $K_0^{MW}(R) = \Z[\red{R^*/(R^*)^2}]/2I_{\Z}$ and $I_{MW} = I_{\Z}/2I_{\Z}$.
Now $I_{\Z}/2I_{\Z}= (\Z/2)^3$ since $I_{\Z}$} has $\Z$ basis the elements $\mylangle u \myrangle$, $1 \neq u \in \red{R^*/(R^*)^2\cong K_4}$.
Hence, the surjection $\red{I_{MW}} \twoheadrightarrow I$, which is $(\Z/2)^3 \twoheadrightarrow (\Z/2)^2$, has kernel $\Z/2$.

As abelian groups, we have $GW(R)\cong \Z \oplus I$ and $K_0^{MW}(R) \cong \Z \oplus \red{I_{MW}}$.
In particular, the computations above  show that $GW(R)\cong \Z \oplus (\Z/2)^2$ and $K_0^{MW}(R) \cong \Z \oplus (\Z/2)^3$.
\end{proof}

\bibliographystyle{plain}

\end{document}